\documentclass[11pt]{amsart}
 \newtheorem{thm}{Theorem}[section]
 \newtheorem{cor}[thm]{Corollary}
 \newtheorem{lem}[thm]{Lemma}
 \newtheorem{prop}[thm]{Proposition}
 \theoremstyle{definition}
 \newtheorem{defn}[thm]{Definition}
 \theoremstyle{remark}
 
 \numberwithin{equation}{section}

 \newcommand{\norm}[1]{\left\Vert#1\right\Vert}
 
 \newcommand{\C}{\mathbb{C}}
 
\begin{document}

\title[]
 {Banach Spaces with respect to Operator-Valued Norms}

\author{ Yun-Su Kim }

%\address{Department of Mathematics, Indiana University, Bloomington,
%Indiana, U.S.A. }

%\email{kimys@indiana.edu}

\address{Yun-Su Kim, Department of Mathematics, University of
Toledo, Toledo, Ohio, 43606, USA }


%\dedicatory{In appreciation of my thesis advisor, Hari Bercovici }

\commby{Daniel J. Rudolph}

%%% ----------------------------------------------------------------------

\begin{abstract}We introduce the notions of $L(H)$-valued norms and Banach spaces with respect to $L(H)$-valued norms.
In particular, we introduce Hilbert spaces with respect to
$L(H)$-valued inner products. In addition, we provide several
fundamental examples of Hilbert spaces with respect to
$L(H)$-valued inner products.

\end{abstract}

%%% ----------------------------------------------------------------------
\maketitle
%%% ----------------------------------------------------------------------

\section*{Introduction}
A nonnegative real-valued function $\norm{\emph{ }}$ defined on a
vector space, called  a \emph{norm}, is a very fundamental and
important project in analysis.\vskip0.1cm

Let $(X,\parallel$ $\parallel_{X})$ be a normed linear space which
is closed relative to the topology induced by the metric defined
by its norm. In this paper, $\textbf{B}(\neq\{0\})$ and
$H(\neq\{0\})$ always denote a Banach space and a Hilbert space,
respectively. A lot of results for scalar-valued functions have
been extended to vector-valued ones. As an example, in \cite{K},
W. L. Paschke introduced an $A$-valued inner product for a
$C^*$-algebra $A$.

In section 1, we provide the notion and fundamental properties of
$L(H)$-valued norms on a normed linear space $X$ by using positive
operators. In theorem \ref{b1}, we characterize $L(H)$-valued
norms preserving Cauchy sequences.

In section 2, we introduce Banach spaces with respect to
$L(H)$-valued norms, and in Theorem \ref{33}, we characterize a
convergent sequence $\{x_{n}\}_{n=0}^{\infty}$ in $X$ with respect
to the norm metric defined by an $L(H)$-valued norm $F$ on $X$.

We discuss a direct sum of Banach spaces with respect to
$L(H)$-valued norms, and prove that it is also a Banach space with
respect to an $L(H)$-valued norm (Theorem \ref{58}). Furthermore,
for a Banach space $\textbf{B}$ with respect to an operator-valued
norm $F:\textbf{B}\rightarrow{L(H)}$, we introduce the notion of
$L(H)$-dual space of $\textbf{B}$, denoted by
$\textbf{B}^{*}_{L(H)}$, and show that there is a natural
$L(H)$-valued norm, called $F_a$, on the quotient space
$\textbf{B}^{*}_{L(H)}/{M_{a}^{\perp}}$ making it into a Banach
space with respect to $F_a$, where
$M_{a}^{\perp}=\{\phi\in{\textbf{B}^{*}_{L(H)}}:\phi(a)=0\}(a\in{\textbf{B}\setminus\{0\}})$
(Theorem \ref{57}).

In section 3, we provide a few of example of Hilbert spaces with
respect to $L(H)$-valued inner product.

As a fundamental example, every Hilbert $L(H)$-module (\cite{C})
is a Hilbert space with respect to an $L(H)$-valued inner product.
In this paper, we discuss about Hilbert spaces with respect to an
$L(H)$-valued inner product, instead of Hilbert modules.

In theorem \ref{59}, we show that any Hilbert space is a Hilbert
space with respect to $L(\C)$-valued inner product, and, in
theorem \ref{22}, $L^{\infty}$ is also a Hilbert space with
respect to an $L(L^{2})$-valued inner product.

Furthermore, in theorem \ref{60}, we prove that $C(Y)$ is a
Hilbert space with respect to an $L(L^{2}(\mu))$-valued inner
product where $(Y,\Omega,\mu)$ is a $\sigma$-finite measure space
and $L^{2}(Y,\Omega,\mu)=L^{2}(\mu)$.

%\newpage
\section{$L(H)$-valued Norms}
%%% ----------------------------------------------------------------------

\subsection{$L(H)$-valued Norms}\label{b2}
Let $\C$ denote the set of complex numbers and $L(\textbf{B})$
denote the set of bounded operators on a Banach space
$\textbf{B}$. For $T\in{L(\textbf{B})}$, the norm of $T$ is
defined by $\norm{T}=\sup_{\norm{x}=1}\norm{Tx}$.

Let $(X,\parallel$ $\parallel_{X})$ be a normed linear space which
is closed relative to the topology induced by the metric defined
by its norm.

 We introduce the notion of
$L(H)$-valued norms on a normed linear space $X$.
%That is, in this section, we
%use positive operators instead of positive numbers, and define an
%operator-valued norm.
To introduce the notion of $L(H)$-valued norms on $X$, we use
positive operators.

Throughout this paper, $H$ will denote a nontrivial Hilbert space
and for any vectors $h_{1}$ and $h_{2}$ in $H$, $(h_{1},h_{2})$
denotes the inner product of $h_{1}$ and $h_{2}$.

\begin{defn}\label{1}

(i) If a function $F:X\rightarrow{L(H)}$ has the following
properties :
\begin{enumerate}

\item For any $x\in{X}$, $F(x)\geq{0}$, i.e. $F(x)$ is a positive
operator.

\item (Triangle Inequality) $F(x+y)\leq{F(x)+F(y)}$ for any $x$
and $y$ in $X$.

\item $F(\lambda{x})=|\lambda|F(x)$ for any $\lambda\in{\C}$.

\item $F(x)=0$ if and only if $x=0$, \end{enumerate} then $F$ is
said to be an $L(H)$-\emph{valued norm } defined on $X$, and
$(X,F)$ is said to be $L(H)$\emph{-valued normed linear space}.

\noindent (ii) If $\sup_{\norm{x}=1}\norm{F(x)}<\infty$, then $F$
is said to be \emph{bounded}.

\noindent (iii) An $L(H)$\emph{-valued metric} on a set $X$ is a
function $d:X\times{X}\rightarrow{L(H)}$ such that for all $x,y$
and $z$ in ${X}$,
\begin{enumerate}
\item $d(x,y)$ is a positive operator,

\item $d(x,y)=0$ if and only if $x=y$,

\item $d(x,y)=d(y,x)$, and

\item $d(x,z)\leq{d(x,y)+d(y,z)}$.

\end{enumerate}

\end{defn}
Let $\mathbb{T}$ be the unit circle in the complex plane. As an
example of such projects, define
$F:L^{\infty}(\mathbb{T})\rightarrow{L(L^{2}(\mathbb{T}))}$ by
\[F(g)=m_{g},\]
where $L^{p}(\mathbb{T})$ is the Lebesgue space with respect to
the Lebesgue measure $\mu$ on $\mathbb{T}$ such that
$\mu(\mathbb{T})=1$, and
$m_{g}:L^{2}(\mathbb{T})\rightarrow{L^{2}(\mathbb{T})}$ is the
bounded operator defined by
\begin{equation}\label{39}m_{g}(f)=|g|\cdot{f},\end{equation}
for $f\in{L^{2}(\mathbb{T})}$. Then clearly, $F$ is an
$L(L^{2}(\mathbb{T}))$-valued norm on $L^{\infty}(\mathbb{T})$.
This multiplication operator $m_{g}$ is a fundamental example.

 If $(X;F)$ is an $L(H)$-valued normed vector
space, the function $d(x,y)=F(x-y)$ is an $L(H)$-valued metric on
$X$.

%Let $F_1$ be a $L(H_{1})$-valued norm on $X$ and $F_{2}$ be a
%$L(H_{2})$-valued norm on $Y$, where $H_1$ and $H_2$ are Hilbert
%spaces. Then we provide a function
%$F_{1}\oplus{F_2}:X\oplus{Y}\rightarrow{L(H_{1}\oplus{H_{2}})}$
%defined by
%$[(F_{1}\oplus{F_{2}})(x\oplus{y})](h_{1}\oplus{h_{2}})=F_{1}(x)h_{1}\oplus{F_{2}(y)h_{2}}$
%for $x\in{X}$, $y\in{Y}$, $h_{1}\in{H_1}$ and $h_{2}\in{H_2}$.

%Then $F_{1}\oplus{F_2}$ is a $L(H_{1}\oplus{H_2})$-valued norm on
%$X\oplus{Y}$.

%Thus for a given $L(\C)$-valued norm $G$ on $X$, we have a
%$\oplus_{i=1}^{n}{L(\C)}$-valued norm, called $(n\times{n})$
%\emph{matrix-valued norm},
%$\oplus_{i=1}^{n}G:\oplus_{i=1}^{n}X\rightarrow{L(\oplus_{i=1}^{n}\C)}$.

\begin{lem}
Let $F:H\rightarrow{L(H)}$ be an $L(H)$-valued norm on $H$ and
$T\in{L(H)}$ be an injective operator.

Then $F\circ{T}$ is an $L(H)$-valued norm on $H$.

\end{lem}
\begin{proof}
Let $G=F\circ{T}$ and $h\in{H}$ be given. Since $F(Th)$ is a
positive operator, $F(Th)=S^{\ast}S$ for some $S\in{L(H)}$. Then
\begin{center}$(G(h)k,k)=\norm{Sk}^{2}\geq{0}$ for $k\in{K}$.\end{center} Thus $G(h)\geq{0}$ for any
$h\in{H}$.

Since $T$ is linear and $F$ is an $L(H)$-valued norm on $H$,
$G(h+k)=F(Th+Tk)\leq{F(Th)+F(Tk)}={G(h)+G(k)}$ for $h$ and $k$ in
$H$, and $G(\lambda{h})=F(\lambda{Th})=|\lambda|G(h)$ for any
$\lambda\in{\C}$.

Let $G(x)=0$ for some $x\in{H}$. Then $F(Tx)=0$ and so $Tx=0$,
since $F$ is an $L(H)$-valued norm. The injectivity of $T$ implies
that $x=0$. Conversely, if $x=0$, then clearly $G(x)=0$.
Therefore, $G=F\circ{T}$ is an $L(H)$-valued norm.

\end{proof}

%\begin{prop}\label{6}\cite{R}
%If $T\in{L(H)}$ is normal, then
%\begin{center}{$\norm{T}=\sup\{|(Th,h)|:h\in{H},
%\norm{h}\leq{1}\}.$}\end{center}
%\end{prop}

\begin{prop}\label{56}
Let $F:X\rightarrow{L(H)}$ be an $L(H)$-valued norm on $X$. Then
\begin{enumerate}
\item $\norm{F(x+y)}\leq\norm{F(x)}+\norm{F(y)}$ for any $x$ and
$y$ in $X$.

\item $\norm{F(x)-F(y)}\leq{\norm{F(x-y)}}$ for any $x$ and $y$ in
$X$.\end{enumerate}
\end{prop}
\begin{proof}
(1) For any $x$ and $y$ in $X$, by triangle inequality,
$0\leq{F(x+y)}\leq{F(x)+F(y)} $. It follows that
\begin{center}$\norm{F(x+y)}\leq\norm{F(x)+F(y)}\leq\norm{F(x)}+\norm{F(y)}.$\end{center}
\vskip0.2cm

(2) By triangle inequality, for any $x$ and $y$ in $X$,
\begin{center}$F(x)\leq{F(x-y)+F(y)}$ or $F(x)-F(y)\leq{F(x-y)}$.\end{center} Thus for
any $h\in{H}$,
\begin{equation}\label{2}
((F(x)-F(y))h,h)\leq{(F(x-y)h,h)}. \end{equation} Similarly, we
have $F(y)-F(x)\leq{F(y-x)}=|-1|F(x-y)=F(x-y)$. Thus for any
$h\in{H}$,
\begin{equation}\label{3}
((F(y)-F(x))h,h)\leq{(F(x-y)h,h)}.
\end{equation}
Inequalities (\ref{2}) and (\ref{3}) imply that
\begin{equation}\label{4}
|((F(x)-F(y))h,h)|\leq|(F(x-y)h,h)|.
\end{equation}

Since $F(x)$ and $F(y)$ are positive operators, by inequality
(\ref{4}), we conclude that
\begin{center}$\norm{F(x)-F(y)}\leq\norm{F(x-y)}.$\end{center}
\end{proof}
\begin{prop}\label{7}
Let $F:X\rightarrow{L(H)}$ be an $L(H)$-valued norm on $X$. Then
$F$ is continuous at 0 if and only if $F$ is a continuous function
on $X$.
\end{prop}
\begin{proof}
Suppose that $F$ is continuous at 0. Let $x\in{X}$ and
$\{x_{\alpha}\}_{\alpha\in{A}}$ be a net in $X$ converging to $x$.
%\begin{equation}\label{a1}\norm{x_{n}-x}_{X}\rightarrow{0}\end{equation} as
%$n\rightarrow\infty$.
By Proposition \ref{56}, for any $\alpha\in{A}$,
\begin{equation}\label{a}
\norm{F(x_{\alpha})-F(x)}\leq\norm{F(x_{\alpha}-x)}.\end{equation}

Since $F$ is continuous at 0 and $F(0)=0$, (\ref{a}) implies that
%\begin{center}$\lim_{n\rightarrow\infty}{F(x_{n}-x)=F(0)=0}$.\end{center} Thus (\ref{a})
%shows that
$\{F(x_{\alpha})\}_{\alpha\in{A}}$ converges to $F(x)$.
%\end{center} as
%$n\rightarrow\infty$. Thus $F$ is continuous at $x\in{X}$.
Since $x\in{X}$ is an arbitrary point in $X$, we conclude that $F$
is continuous on $X$. The converse is clear.
%If $F$ is a continuous function
%on $X$, evidently $F$ is continuous at 0.
\end{proof}
\subsection{\textbf{$L(H)$-valued Norms
Preserving Cauchy Sequences}}\label{b}
%Many of well-known function spaces are Banach spaces.
In Section \ref{b2}, we provided an $L(H)$-valued norm and, in
this section, we provide a notion of $L(H)$-valued norms
preserving Cauchy sequences.
\begin{defn}\label{b}
Let $\{x_{n}\}_{n=1}^{\infty}$ be a Cauchy sequence in $X$ and
$F:X\rightarrow{L(H)}$ be an $L(H)$-valued norm on $X$.

If $\{F(x_{n})\}_{n=1}^{\infty}$ is also a Cauchy sequence in
$L(H)$, then $F$ is said to be an $L(H)$-valued norm
\emph{preserving Cauchy sequences}.

%$X$ is said to be \emph{complete with respect to} $F$ or \emph{a
%Banach space with respect to }$F$.
\end{defn}

\begin{thm}\label{b1}
Let $F:X\rightarrow{L(H)}$ be a bounded $L(H)$-valued norm on $X$.
Then, $F$ preserves Cauchy sequences
 if and only if $F$ is continuous at 0.
\end{thm}
\begin{proof}
Suppose that $F$ preserves Cauchy sequences and
$\{x_{n}\}_{n=1}^{\infty}$ is a sequence in $X$ such that
\begin{equation}\label{s1}\lim_{n \rightarrow \infty}{x_{n}}=0.\end{equation} Then $\{x_{n}\}_{n=1}^{\infty}$ is a Cauchy
sequence in $X$. Since $F$ preserves Cauchy sequences,
$\{F(x_{n})\}_{n=1}^{\infty}$ is also a Cauchy sequence.
%that is,
%for a given $\epsilon>0$, there is a natural number $N(\epsilon)$
%such that all natural numbers $n,m\geq{N(\epsilon)}$, we have
%\begin{center}$\norm{F(x_{n})-F(x_{m})}<\epsilon$.\end{center}
Because $L(H)$ is a Banach space, there is an operator
$T\in{L(H)}$ such that
\begin{equation}\label{s}\lim_{n \rightarrow \infty}{F(x_{n})}=T.\end{equation}

If $A=\{n:x_{n}=0\}$ is infinite, then for $n_{k}\in{A}$
$(k=1,2,3,\cdot\cdot\cdot)$, \begin{center}$\lim_{k \rightarrow
\infty}{F(x_{n_{k}})}=F(0)=0$\end{center} and so by equation
(\ref{s}), $T=0$. \vskip0.2cm If $A=\{n:x_{n}=0\}$ is not
infinite, then there is a nonzero subsequence
$\{y_{n}\}_{n=1}^{\infty}$ of $\{x_{n}\}_{n=1}^{\infty}$ such that
\begin{equation}\label{s2}\lim_{n \rightarrow \infty}{y_{n}}=0\texttt{ and }\lim_{n \rightarrow \infty}{F(y_{n})}=T.
\end{equation}
%\begin{equation}\label{s5}\lim_{n \rightarrow \infty}{y_{n}}=0.\end{equation}
By equation (\ref{s2}),
\begin{equation}\label{s3}\lim_{n \rightarrow
\infty}\norm{F(y_{n})}=\norm{T}.\end{equation} Let
$M_{F}=\sup_{\norm{x}=1}{\norm{F(x)}}(<\infty)$ (Note that $F$ is
bounded). By property (3) of Definition \ref{1} and equation
(\ref{s3}), $\norm{T}=\lim_{n \rightarrow
\infty}\norm{y_{n}}_{X}\norm{F(\frac{y_{n}}{\norm{y_{n}}_{X}})}\leq\lim_{n\rightarrow
\infty}\norm{y_{n}}_{X}\cdot{M_F}=0$.

%\noindent If $M_{F}=\sup_{\norm{x}=1}{\norm{F(x)}}$, we get
%\begin{equation}\label{s4} \norm{T}=\lim_{n \rightarrow
%\infty}\norm{y_{n}}_{X}\norm{F(\frac{y_{n}}{\norm{y_{n}}_{X}})}\leq\lim_{n
%\rightarrow \infty}\norm{y_{n}}_{X}\cdot{M_{F}}.
%\end{equation}
%Since  $F$ is a bounded operator-valued norm on $X$,
%\begin{center}$M_{F}=\sup_{\norm{x}=1}{\norm{F(x)}}$ $<\infty$.\end{center}
%By (\ref{s5}) and
%inequality (\ref{s4}), we conclude that
It follows that \begin{center}$\norm{T}=0$\end{center} and so
$T=0$. Thus $T\equiv{0}$ whether $A$ is infinite or not.

\noindent By equation (\ref{s}), $\lim_{n \rightarrow
\infty}{F(x_{n})}=0=F(0)$ which proves that $F$ is continuous at
0. \vskip0.2cm
%Let $\epsilon>0$ be given. Since $F$ is continuous at 0, there is
%a $\delta(\epsilon)>0$ such that if
%$\norm{x}_{X}<\delta(\epsilon)$, then
%\begin{equation}\label{a4}\norm{F(x)}<\epsilon.\end{equation}

%Let $\{x_{n}\}_{n=1}^{\infty}$ be a Cauchy sequence in $X$. Then
%there is a natural number $N(\epsilon)$ such that for all natural
%numbers $n,m\geq{N(\epsilon)}$,
%\begin{center}$\norm{x_{n}-x_{m}}_{X}<\delta(\epsilon)$.\end{center} Inequality (\ref{a4}) implies
%that
%\begin{equation}\label{a5}\norm{F(x_{n}-x_{m})}<\epsilon.\end{equation}
Conversely, suppose that $F$ is continuous at 0. By Proposition
\ref{56} (2), it is easy to see that $F$ preserves Cauchy
sequences.
\end{proof}
From Proposition \ref{7}, we obtain the following result:
\begin{cor}
Let $F:X\rightarrow{L(H)}$ be a bounded $L(H)$-valued norm on $X$.
Then, $F$ preserves Cauchy sequences if and only if $F$ is
continuous on $X$.
\end{cor}

\section{Banach Spaces with respect to $L(H)$-Valued Norms }
In this paper, $(X,F)$ always denote an $L(H)$-valued normed
linear space.

\begin{defn}
A sequence $\{x_{n}\}_{n=0}^\infty$ in $X$ is said to be a
\emph{Cauchy sequence with respect to F} if for every
$\epsilon>0$, there is a natural number $N(\epsilon)$ such that
for all $n,m\geq{N(\epsilon)}$, we have
\begin{equation}\label{30}F(x_{n}-x_{m})\leq\epsilon{I_{H}}\end{equation} where $I_H$ is the identity
operator on $H$. If $(X;F)$ is a $L(H)$-valued normed vector
space, then the function $d(x,y)=F(x-y)$ is a $L(H)$-valued metric
on $X$.\end{defn}

\begin{prop}\label{31}
Let $\{x_{n}\}_{n=0}^\infty$ be a sequence in an $L(H)$-valued
normed linear space $X$. Then the following statements are
equivalent;

(a) The sequence $\{x_{n}\}_{n=0}^\infty$ is a Cauchy sequence.

(b) For every $\epsilon>0$, there is a natural number
$N(\epsilon)$ such that for all $n,m\geq{N(\epsilon)}$, we have
$\norm{F(x_{n}-x_{m})}\leq\epsilon$.
\end{prop}
\begin{proof}
$(a)\Rightarrow(b)$ Since $F(x)\geq{0}$ for $x\in{X}$,
\begin{equation}\label{29}\norm{F(x)}=\texttt{sup}\{(F(x)h,h);h\in{H},
\norm{h}={1}\}.
\end{equation}
If $\{x_{n}\}_{n=0}^\infty$ is a Cauchy sequence, for a given
$\epsilon>0$, there is a natural number $N(\epsilon)$ such that
for all $n,m\geq{N(\epsilon)}$, the inequality (\ref{30}) is true.
It follows that $(F(x_{n}-x_{m})h,h)\leq\epsilon({I_{H}}h,h)$ for
all $n,m\geq{N(\epsilon)}$. By equation (\ref{29}),
$\norm{F(x_{n}-x_{m})}\leq\epsilon$ for all
$n,m\geq{N(\epsilon)}$. \vskip0.1cm

$(b)\Rightarrow(a)$ If $(b)$ is true, by equation (\ref{29}),
\[(F(x_{n}-x_{m})h,h)\leq\epsilon({I_{H}}h,h),\]
for all $n,m\geq{N(\epsilon)}$ and $\norm{h}=1$. It follows that
$F(x_{n}-x_{m})\leq\epsilon{I_{H}}$ for all
$n,m\geq{N(\epsilon)}$.

\end{proof}

\begin{defn}
(a) $X$ is said to be \emph{complete with respect to }$F$ if every
Cauchy sequence (with respect to $F$) $\{x_{n}\}_{n=0}^{\infty}$
in $X$ converges to an element $x$ of $X$ with respect to the norm
metric defined by $F$, that is, for a given $\epsilon>0$, there is
a natural number $N(\epsilon)$ such that
\[{F(x_{n}-x)}\leq\epsilon{I_{H}}\] for any $n\geq{N(\epsilon)}$.
%\end{defn}

(b) A complex linear space $X$ with an $L(H)$-valued norm
$F:X\rightarrow{L(H)}$ such that $X$ is complete with respect to
$F$ is called a \emph{Banach space with respect to }$F$.

\end{defn}

By the same proof as Proposition \ref{31}, we have the following
useful result.

\begin{thm}\label{33}
Let $\{x_{n}\}_{n=0}^\infty$ be a sequence in an $L(H)$-valued
normed linear space $X$. Then the following statements are
equivalent;

(a) The sequence $\{x_{n}\}_{n=0}^{\infty}$ converges to an
element $x$ of $X$ with respect to the norm metric defined by $F$.

(b) For a given $\epsilon>0$, there is a natural number
$N(\epsilon)$ such that
\[\norm{F(x_{n}-x)}\leq\epsilon\] for any $n\geq{N(\epsilon)}$.

\end{thm}
\begin{cor}\label{51}
Let $(X,F)$ be an $L(H)$-valued normed linear space.

If a sequence $\{x_{n}\}_{n=0}^{\infty}$ converges to an element
$x$ of $X$ with respect to the norm metric defined by $F$, then
the sequence $\{F(x_{n})\}_{n=0}^{\infty}$ of operators converges
to $F(x)$.
\end{cor}
\begin{proof}
It is clear because of Proposition \ref{56} and Theorem \ref{33}.
\end{proof}

We now discuss a direct sum of Banach spaces with a natural
operator-valued norm.

\begin{prop}
Let $F_{i}(i=1,2)$ be an $L(H)$-valued norm defined on
$X_{i}(i=1,2,$ respectively$)$. If $F_{1}+{F_{2}}$ is a function
on the algebraic direct sum $X_{1}\oplus{X_2}$ defined by
\begin{equation}\label{27}(F_{1}+{F_{2}})(x_{1}\oplus{x}_{2})=F_{1}(x_{1})+{F_{2}}(x_{2}),
\end{equation}
then $(X_{1}\oplus{X_{2}},F_{1}+F_{2})$ is an $L(H)$-valued normed
linear space.
\end{prop}
\begin{proof}

Clearly, $(F_{1}+F_{2})(x_{1}\oplus{x_{2}})\geq{0}$ for any
$x_{1}\oplus{x_{2}}\in{X_{1}\oplus{X_{2}}}$.

For $x_{1}\oplus{x_{2}}, \texttt{
}y_{1}\oplus{y_{2}}\in{X_{1}\oplus{X_{2}}}$,
$(F_{1}+F_{2})(x_{1}\oplus{x_{2}}+y_{1}\oplus{y_{2}})=F_{1}({x_{1}}+y_{1})+F_{2}(x_{2}+y_{2})\leq
{F_{1}({x_{1}})+F_{1}(y_{1})+F_{2}(x_{2})+F_{2}(y_{2})}=(F_{1}+F_{2})(x_{1}\oplus{x_{2}})+(F_{1}+F_{2})
(y_{1}\oplus{y_{2}})$, that is, triangle inequality holds.

Next,
$(F_{1}+F_{2})(\lambda(x_{1}\oplus{x_{2}}))=|\lambda|(F_{1}+F_{2})(x_{1}\oplus{x_{2}})$
for any $\lambda\in{\C}$.

Finally, $(F_{1}+F_{2})(x_{1}\oplus{x_{2}})=0$ if and only if
$x_{1}=x_{2}=0$ if and only if $x_{1}\oplus{x_{2}}=0$.

\end{proof}

\begin{thm}\label{58}
Let $X_{i}(i=1,2)$ be a Banach space with respect to
$F_{i}:X_{i}\rightarrow{L(H)}$. Then $X_{1}\oplus{X_2}$ is a
Banach space with respect to $F_{1}+{F_{2}}$.

\end{thm}
\begin{proof}
Let $\{x_{n}\oplus{y_{n}}\}_{n=1}^{\infty}$ be a Cauchy sequence
in $(X_{1}\oplus{X_2},F_{1}+{F_{2}})$. Then, for a given
$\epsilon>0$, there is a natural number $N(\epsilon)$ such that
\begin{equation}\label{34}{{(F_{1}+{F_{2}})(x_{n}\oplus{y_{n}}-x_{m}\oplus{y_{m}})}}\leq\epsilon{I_{H}}\end{equation} for
any $n, m\geq{N(\epsilon)}$. Since
$(F_{1}+{F_{2}})(x_{n}\oplus{y_{n}}-x_{m}\oplus{y_{m}})=F_{1}(x_{n}-x_{m})+F_{2}(y_{n}-y_{m})
$,
\begin{equation}\label{47}{F_{1}(x_{n}-x_{m})}\leq\epsilon{I_{H}}\texttt{ and }{F_{2}(y_{n}-y_{m})}\leq\epsilon{I_{H}}\end{equation}
 for
any $n, m\geq{N(\epsilon)}$.

Thus, $\{x_{n}\}_{n=1}^{\infty}$ and $\{{y_{n}}\}_{n=1}^{\infty} $
are Cauchy sequences in $X_{1}$ and $X_{2}$, respectively. Since
$X_{i}(i=1,2)$ is a Banach space with respect to $F_{i}$, there
are $x(\in{X_{1}})$ and $y(\in{X_{2}})$ such that
$\{x_{n}\}_{n=1}^{\infty}$ and $\{{y_{n}}\}_{n=1}^{\infty} $
converges to $x$ and $y$ with respect to $F_1$ and $F_2$,
respectively. Then, clearly,
$\{x_{n}\oplus{y_{n}}\}_{n=1}^{\infty}$ converges to $x\oplus{y}$
with respect to $F_{1}+{F_{2}}$.

\end{proof}
We next define a norm on the space of linear functions from a
Banach space $(X,F)$ to $L(H)$.

\begin{defn}\label{54}

Let $X$ be a Banach space with respect to $F:X\rightarrow{L(H)}$.

(a) If $g:X\rightarrow{L(H)}$ is linear and
\begin{equation}\label{53}\norm{g}_{F}=\sup\{\frac{\norm{g(x)}}{\norm{F(x)}};x\in{X\setminus\{0\}}\}<\infty,\end{equation}
then $g$ is said to be \emph{bounded with respect to }$F$.

(b) The set $X^{*}_{L(H)}=\{g:X\rightarrow{L(H)}:g\texttt{ is
linear and } \norm{g}_{F}<\infty\}$ is said to be
$L(H)$\emph{-dual space of }$X$.

\end{defn}

Note that the equation (\ref{53}) defines a (scalar-valued) norm
on the $L(H)$-dual space $X^{*}_{L(H)}$.

Let \textbf{B} be a Banach space with respect to
$F:\textbf{B}\rightarrow{L(H)}$, and
$a\in{\textbf{B}\setminus\{0\}}$.

If $M_{a}^{\perp}=\{\phi\in{\textbf{B}^{*}_{L(H)}}:\phi(a)=0\}$,
then we want to show that there is a natural $L(H)$-valued norm,
called $F_a(\texttt{equation }(\ref{32}))$, on the quotient space
$\textbf{B}^{*}_{L(H)}/{M_{a}^{\perp}}$ making it into a Banach
space with respect to $F_a$.

\begin{prop}\label{12}
Let \textbf{B} be a Banach space with respect to
$F:\textbf{B}\rightarrow{L(H)}$, and
$a\in{\textbf{B}\setminus\{0\}}$.

If $M_{a}^{\perp}=\{\phi\in{\textbf{B}^{*}_{L(H)}}:\phi(a)=0\}$,
then $(\textbf{B}^{*}_{L(H)}/{M_{a}^{\perp}},F_{a})$ is an
$L(H)$-valued normed space where
$F_{a}:\textbf{B}^{*}_{L(H)}/{M_{a}^{\perp}}\rightarrow{L(H)}$ is
defined by
\begin{equation}\label{32}F_{a}(\varphi)=\norm{\varphi(a)}{F(a)}.
\end{equation}
\end{prop}
\begin{proof}
If $\phi=\psi\in{\textbf{B}^{*}_{L(H)}/{M_{a}^{\perp}}}$, then
$\phi(a)=\psi(a)$. Thus, $F_{a}$ is well-defined.
 Clearly,
$F_{a}(\varphi)\geq{0}$ for any
$\varphi\in{\textbf{B}^{*}_{L(H)}/{M_{a}^{\perp}}}$, and
$F_{a}(\lambda\varphi)=|\lambda|F_{a}(\varphi)$ for any
$\lambda\in\C$. For $\varphi$ and $\psi$ in
$\textbf{B}^{*}_{L(H)}/{M_{a}^{\perp}}$,
\begin{center}$F_{a}(\varphi+\psi)\leq\norm{\varphi(a)}F(a)+\norm{\psi(a)}F(a)=
F_{a}(\varphi)+F_{a}(\psi)$.\end{center}

Finally, if $F_{a}(\varphi)=0$, then $\varphi(a)=0$ or $F(a)=0$.
If $\varphi(a)=0$, then $\varphi\equiv{0}$ in
$\textbf{B}^{*}/{M_{a}^{\perp}}$. Since $F$ is an $L(H)$-valued
norm, if $F(a)=0$, then $a=0$ which contradict our assumption.
Conversely, if $\varphi=0$ in $\textbf{B}^{*}/{M_{a}^{\perp}}$,
then clearly, $F_{a}(\varphi)=0$.

\end{proof}

\begin{thm}\label{57}
%Let \textbf{B} be a Banach space with respect to
%$F:\textbf{B}\rightarrow{L(H)}$, and
%$a\in{\textbf{B}\setminus\{0\}}$.

In the same assumption as Proposition \ref{12},
$\textbf{B}^{*}_{L(H)}/{M_{a}^{\perp}}$ is a Banach space with
respect to the operator-valued norm
$F_{a}:\textbf{B}^{*}_{L(H)}/{M_{a}^{\perp}}\rightarrow{L(H)}$.
% defined by
%\begin{equation}\label{32}F_{a}(\varphi)=\mid\varphi(a)\mid{F(a)}.
%\end{equation}
\end{thm}
\begin{proof}
Let $\epsilon>0$ be given and $\{\varphi_{n}\}_{n=0}^\infty$ be a
Cauchy sequence with respect to $F_{a}$ in
$\textbf{B}^{*}_{L(H)}/{M_{a}^{\perp}}$.

Then, there is a natural number $N(\epsilon)$ such that for any
$n,m\geq{N(\epsilon)}$,
\begin{equation}\label{35}\norm{F_{a}(\varphi_{n}-\varphi_{m})}_{L(H)}=
\norm{\varphi_{n}(a)-\varphi_{m}(a)}\norm{F(a)}_{L(H)}\leq\epsilon.\end{equation}
%$F_{a}(\varphi)=\mid\varphi(a)\mid{F(a)}$
Since $F(a)\neq{0}$ (Note that $a\neq{0}$), equation (\ref{35})
implies that there is an operator $T_{a}\in{L(H)}$ such that
$\lim_{n\rightarrow\infty}\varphi_{n}(a)=T_{a}$. Since there
exists $\varphi\in{B^{*}_{L(H)}}$ such that $\varphi(a)=T_{a}$,
\begin{equation}\label{36}\lim_{n\rightarrow\infty}F_{a}(\varphi_{n}-\varphi)=
\lim_{n\rightarrow\infty}\norm{\varphi_{n}(a)-\varphi(a)}F(a)=0.
\end{equation}
Thus, the Cauchy sequence $\{\varphi_{n}\}_{n=0}^\infty$ in
$\textbf{B}^{*}_{L(H)}/{M_{a}^{\perp}}$ converges to an element
$\varphi(\in{\textbf{B}^{*}_{L(H)}/{M_{a}^{\perp}}})$ with respect
to the norm metric defined by $F_{a}$.

\end{proof}

\section{Hilbert Spaces with respect to $L(H)$-Valued Inner Product }
\subsection{An $L(H)$-Valued Inner Product}

In \cite{P}, W. L. Paschke introduced an operator-valued inner
product. Recall that for any operator $T$ in $L(H)$, the adjoint
of $T$, denoted $T^*$ is the unique operator on $H$ satisfying
$(Tf,g)=(f,T^{*}g)$ for $f$ and $g$ in $H$.

\begin{defn}
An $L(H)$\emph{-valued inner product} on a complex linear space
$X$ is a conjugate-bilinear map
$\varphi:X\times{X}\rightarrow{L(H)}$ such that:\vskip0.2cm

(1) $\varphi(x,x)\geq{0}$ for any $x\in{X};$

(2) $\varphi(x,x)=0$ if and only if $x=0;$

(3) $\varphi(x,y)=\varphi(y,x)^{*}$ for any $x,y\in{X}$.

\end{defn}

%In this paper, we will discuss about.

\subsection{Hilbert Space with respect to an $L(H)$-Valued Inner Product }

As scalar-valued cases, we provide the notion of a Hilbert space
with respect to an $L(H)$-valued inner product.

Let $X$ be a complex linear space with an $L(H)$-valued inner
product $\varphi:X\times{X}\rightarrow{L(H)}$. Then a function
$F_{\varphi}:X\rightarrow{L(H)}$, defined by
\begin{equation}\label{37} F_{\varphi}(x)=[\varphi(x,x)]^{1/2}\texttt{ for }x\in{X},
\end{equation} could not be an $L(H)$-valued norm. Thus, we need an
assumption to provide a notion of a \emph{Hilbert Space with
respect to }an $L(H)$\emph{-Valued Inner Product}.

\begin{defn}
Let $X$ be a complex linear space with an $L(H)$-valued inner
product $\varphi:X\times{X}\rightarrow{L(H)}$. If a function
$F_{\varphi}:X\rightarrow{L(H)}$ defined by
$F_{\varphi}(x)=\varphi(x,x)^{1/2}(x\in{X})$ is an $L(H)$-valued
norm, and $X$ is complete with respect to $F_{\varphi}$, then $X$
is called a \emph{Hilbert space with respect to }$\varphi$.

\end{defn}
%From now on we will indicate $\varphi(x,x)$ by
%$\langle{x,x}\rangle$ for any $x\in{X}$.
As an example, every Hilbert $L(H)$-module (\cite{C}) is a Hilbert
space with respect to an $L(H)$-valued inner product. In this
paper, we discuss about Hilbert spaces with respect to an
$L(H)$-valued inner product, instead of Hilbert modules.

The Cauchy-Schwarz inequality is useful in the study of
(scalar-valued) inner product spaces. To extend the version of
Cauchy-Schwarz inequality, we define the notion of \emph{bounded}
$L(H)$\emph{-valued inner product}.
\begin{defn}
Let $\varphi:X\times{X}\rightarrow{L(H)}$ be an $L(H)$-valued
inner product, and $F_{\varphi}:X\rightarrow{L(H)}$ be an
$L(H)$-valued norm defined by equation (\ref{37}).

(a) If there exists a positive number $M$ such that for any
$x,y\in{X}$,
\begin{equation}\label{52}\norm{\varphi(x,y)}\leq{M}{\norm{F_{\varphi}(x)}}{\norm{F_{\varphi}(y)}},\end{equation} then
$\varphi$ is said to be a \emph{bounded} $L(H)$\emph{-valued inner
product}.

(b) If $\varphi$ is bounded, then
$\norm{\varphi}_{F_{\varphi}}=\sup\{\frac{\norm{\varphi(x,y)}}{\norm{F_{\varphi}(x)}\norm{F_{\varphi}(y)}}:x(\neq{0}),y(\neq{0})\in{X}\}$
is said to be a \emph{norm of} $\varphi$ \emph{with respect to}
$F_\varphi$.

\end{defn}
Since $F_{\varphi}(x)\geq{0}$,
$\norm{F_{\varphi}(x)^{2}}=\norm{F_{\varphi}(x)}^2=\norm{\varphi(x,x)}$.
It follows the following lemma;
\begin{lem}
If $(X,\varphi)$ is a bounded $L(H)$-valued inner product space,
then $\norm{\varphi}_{F_{\varphi}}\geq{1}$.
\end{lem}
For given sequences $\{x_{n}\}_{n=0}^{\infty}$ and
$\{y_{n}\}_{n=0}^{\infty}$ in an $L(H)$-valued inner product space
$(X,\varphi)$ with limits $x$ and $y$ (with respect to the norm
metric defined by $F_{\varphi}$ defined in equation (\ref{37})),
by Corollary \ref{51},
$\sup\{\norm{F_{\varphi}(x_{n})}:n=0,1,2,\cdot\cdot\cdot\}$ and
$\sup\{\norm{F_{\varphi}(y_{n})}:n=0,1,2,\cdot\cdot\cdot\}$ are
finite. Thus, if $\varphi$ is bounded, then, by Theorem \ref{33}
and (\ref{52}),
$\lim_{n\rightarrow\infty}\varphi(x_{n},y_{n})=\varphi(x,y)$.

In Definition \ref{54}, we provided the notion of $L(H)$-dual
spaces of a Banach space $X$ with respect to an operator-valued
norm.  We now introduce an example of an element in
$X^{*}_{L(H)}$. Let $(X,\varphi)$ be a bounded $L(H)$-valued inner
product space, and $\varphi_{y}:X\rightarrow{L(H)}$ be a linear
map defined by $\varphi_{y}(x)=\varphi(x,y)$. Then
$\norm{\varphi_{y}}_{F_{\varphi}}=\sup\{\frac{\norm{\varphi(x,y)}}{\norm{F_{\varphi}(x)}}:
x\in{X\setminus\{0\}}\}\leq{\norm{\varphi}_{F}\norm{F_{\varphi}(y)}}$,
that is, $\norm{\varphi_{y}}_{F}$ is finite.

\subsection{Examples of Hilbert Spaces with respect to Operator-Valued Inner Products.}
In this section, we provide several examples of Hilbert Spaces
with respect to Operator-Valued Inner Products.

We state the definition of $L(H)$\emph{-dual with respect to}
$L(H)$\emph{-valued inner product}, when a Hilbert space $X$ with
respect to $\varphi:X\times{X}\rightarrow{L(H)}$ is also a
(scalar-valued) Hilbert space.

%If a Hilbert space $H$ is a Hilbert space with respect to
%$\varphi:H\times{H}\rightarrow{L(H)}$.
 \begin{defn}

Let $H$ be a Hilbert space with respect to an $L(H)$-valued inner
product $\varphi:H\times{H}\rightarrow{L(H)}$.

%(a) If $g:H\rightarrow{L(H)}$ is linear and
%\begin{equation}\label{55}\norm{g}_{F}=\sup\{\frac{\norm{g(x)}}{\norm{F(x)}};x\in{X\setminus\{0\}}\}<\infty,\end{equation}
%then $g$ is said to be \emph{bounded with respect to }$F$.

 The set $H^{*}_{L(H)}$ consisting of a linear map $g:H\rightarrow{L(H)}$ such
 that

 (a) $ \norm{g}_{F_{\varphi}}<\infty$,  and

 (b) $\phi(h)\varphi(x,y)=\varphi(\phi(h)x,y)$,

  is said to be $L(H)$\emph{-dual space of }$H$ with respect to
  $\varphi$.

\end{defn}
 It will be denoted by $(H^{*}_{L(H)},\varphi)$.

By the Riesze representation theorem, we can identify the space
$\C$ of complex numbers with $L(\C)$, by the map
$g\rightarrow\psi_{g}$ for $g$ in $\C$, where
$\psi_{g}:\C\rightarrow\C$ is a linear function defined by
\begin{equation}\label{24}\psi_{g}(f)=(f,g).\end{equation}

If $\varphi:\C\times{\C}\rightarrow{L(\C)}$ is a
conjugate-bilinear map defined by
\begin{equation}\label{26}
\varphi(a,b)=\psi_{a\overline{b}}
\end{equation}
for $a$ and $b$ in $\C$. The space $\C$ of complex numbers is a
Hilbert space with respect to $\varphi$ defined in equation
(\ref{26}). We can see easily that the $L(\C)$-dual space of $\C$
with respect to
  $\varphi$ is $L(\C)$. As the first example,
\begin{thm}\label{59}
Let $H$ be a Hilbert space, and
$\phi:H\times{H}\rightarrow{L(\C)}$ be a conjugate-bilinear map
defined by
\begin{equation}\label{2}
\phi(a,b)=\psi_{(a,{b})}
\end{equation}
for $a$ and $b$ in $H$.

Then, $H$ is a Hilbert space with respect to $\phi$.

\end{thm}

\begin{proof}
Clearly, $\phi$ is a bounded $L(\C)$-valued inner product. We now
show that a function $F_{\phi}:H\rightarrow{L(\C)}$ is an
$L(\C)$-valued norm.

Since $F_{\phi}(f)=\psi_{|f|}$,
\begin{equation}\label{3}(\psi_{|f|}z,z)=|f||z|^{2}\geq{0}\texttt{ for any
}z\in\C.\end{equation} Clearly,
$F_{\phi}(f+g)\leq{F_{\phi}(f)+F_{\phi}(g)}$ for any $f$ and $g$
in $H$, and $F_{\phi}(\lambda{f})=|\lambda|F_{\phi}(f)$ for any
$\lambda\in{\C}$. Finally, $F_{\phi}(f)=0$ if and only if $f=0$.

Thus, $F_{\phi}$ is an $L(\C)$-valued norm.

Let a sequence $\{f_{n}\}_{n=0}^{\infty}$ in $H$ be a Cauchy
sequence with respect to $F_\phi$. Then, for a given $\epsilon>0$,
there is a natural number $N(\epsilon)$ such that for all
$n,m\geq{N(\epsilon)}$,
\begin{equation}\label{17}F_{\phi}(f_{n}-f_{m})=\psi_{|f_{n}-f_{m}|}\leq\epsilon{I_{\C}}\end{equation}
where $I_{\C}$ is the identity operator on $\C$.

Since
$\norm{\psi_{|f_{n}-f_{m}|}}=\norm{f_{n}-f_{m}}_{H}\leq\epsilon$
for all $n,m\geq{N(\epsilon)}$, $\{f_{n}\}_{n=0}^{\infty}$ is a
Cauchy sequence in the space $H$. Thus, there is an element $f$ in
$H$ and a natural number $N^{\prime}(\epsilon)$ such that
$\norm{f_{n}-f}_{H}\leq\epsilon$ if $n\geq{N^{\prime}(\epsilon)}$.

It follows that
\[\norm{F_{\phi}(f_{n}-f)}=\norm{\psi_{|f_{n}-f|}}=\norm{f_{n}-f}_{H}\leq\epsilon,\]
if $n\geq{N^{\prime}}(\epsilon)$. Therefore, the sequence
$\{f_n\}_{n=0}^{\infty}$ converges to the function $f$ in $H$ with
respect to $F_\phi$, that is, $H$ is complete with respect to
$F_\phi$.

\end{proof}

 We know that
$L^{\infty}(\mathbb{T})=L^{\infty}$ is a Banach space, but it is
not a Hilbert space. However, we prove that $L^{\infty}$ is a
Hilbert space with respect to an
$L(L^{2}(\mathbb{T}))(=L(L^{2}))$-valued inner product.

\begin{prop}\label{38}
Let $\varphi:L^{\infty}\times{L^{\infty}}\rightarrow{L(L^{2})}$ be
a function defined by
\begin{equation}\label{40}\varphi(f,g)=M_{f\overline{g}}\end{equation}
where $M_{g}:L^{2}\rightarrow{L^{2}}$ is the multiplication
operator defined by
\begin{equation}\label{46}M_{g}(f)=g\cdot{f},\end{equation}
for $f\in{L^{2}}$. Then, $\varphi$ is a bounded $L(L^{2})$-valued
inner product.

\end{prop}
\begin{proof}
For $\alpha_1$, $\alpha_2$ in $\C$ and $f_{i}$, $g_{i}(i=1,2)$ in
$L^{\infty}$, clearly,
\begin{equation}\label{41}
\varphi(\alpha_{1}f_{1}+\alpha_{2}f_{2},g)=\alpha_{1}\varphi(f_{1},g)+\alpha_{2}\varphi(f_{2},g)
\end{equation}
and
\begin{equation}\label{42}
\varphi(f,\alpha_{1}g_{1}+\alpha_{2}g_{2})=\overline{\alpha_{1}}\varphi(f,g_{1})+\overline{\alpha_{2}}\varphi(f,g_{2}).
\end{equation}

Since the operator $M_{|g|^{2}}$ is positive,
\begin{equation}\label{43}\varphi(g,g)\geq{0}\end{equation}
for any $g$ in $L^{\infty}$, and
\begin{equation}\label{44}\varphi(g,g)=0\texttt{ if and only if }g\equiv{0}.
\end{equation}

Since $M_{f\overline{g}}={M_{\overline{f}g}}^{*}$,
\begin{equation}\label{45}\varphi(f,g)=\varphi(g,f)^{*}.
\end{equation}

By (\ref{41}), (\ref{42}), (\ref{43}), (\ref{44}), and (\ref{45}),
$\varphi$ is an $L(L^{2})$-valued inner product.

Since
$\norm{\varphi(f,g)}=\norm{M_{f\overline{g}}}=\norm{f\overline{g}}_{\infty}\leq{\norm{f}_{\infty}}
\norm{g}_{\infty}=\norm{F(f)}\norm{F(g)}$, $\varphi$ is bounded.
In fact, $\norm{\varphi}_{F_\varphi}=1$

\end{proof}

\begin{thm}\label{22}
Let $\varphi$ be the bounded $L(L^{2})$-valued inner product
defined in Proposition \ref{38}.

Then, $L^{\infty}$ is a Hilbert space with respect to $\varphi$.
\end{thm}
\begin{proof}To start proving that $L^{\infty}$ is a Hilbert space with
respect to $\varphi$, we need to show that a function
$F_{\varphi}:L^{\infty}\rightarrow{L(L^{2})}$ defined by the same
way as equation (\ref{37}) is an $L(L^{2})$-valued norm.

Since $F_{\varphi}(f)^{2}=M_{|f|^{2}}$,
$F_{\varphi}(f)=M_{|f|}=m_{f}$ (Note equation (\ref{39})). Thus,
$F_{\varphi}$ is an $L(L^{2})$-valued norm.

Let a sequence $\{f_{n}\}_{n=0}^{\infty}$ in $L^{\infty}$ be a
Cauchy sequence with respect to $F_\varphi$. Then, for a given
$\epsilon>0$, there is a natural number $N(\epsilon)$ such that
for all $n,m\geq{N(\epsilon)}$,
\begin{equation}\label{17}F_{\varphi}(f_{n}-f_{m})=M_{|f_{n}-f_{m}|}\leq\epsilon{I_{L^{2}}}\end{equation}
where $I_{L^{2}}$ is the identity operator on $L^{2}$.

Since
$\norm{M_{|f_{n}-f_{m}|}}=\norm{f_{n}-f_{m}}_{\infty}\leq\epsilon$
for all $n,m\geq{N(\epsilon)}$, $\{f_{n}\}_{n=0}^{\infty}$ is a
Cauchy sequence in the space $L^{\infty}$. Thus, there is a
function $f$ in $L^\infty$ and a natural number
$N^{\prime}(\epsilon)$ such that
$\norm{f_{n}-f}_{\infty}\leq\epsilon$ if
$n\geq{N^{\prime}(\epsilon)}$.

It follows that
\[\norm{F_{\varphi}(f_{n}-f)}=\norm{M_{|f_{n}-f|}}=\norm{f_{n}-f}_{\infty}\leq\epsilon,\]
if $n\geq{N^{\prime}}(\epsilon)$. Therefore, the sequence
$\{f_n\}_{n=0}^{\infty}$ converges to the function $f$ in
$L^\infty$ with respect to $F_\varphi$, that is, $L^{\infty}$ is
complete with respect to $F_\varphi$.

\end{proof}

By the same proof of Theorem \ref{22}, we have the following
result.
\begin{cor}\label{21}
Let $(Y,\Omega,\mu)$ be a $\sigma$-finite measure space and
$L^{p}(Y,\Omega,\mu)=L^{p}(\mu)$ $(p=2,\infty)$.

Then, $L^{\infty}(\mu)$ is a Hilbert space with respect to a
bounded $L(L^{2}(\mu))$-valued inner product
$\varphi:Y\times{Y}\rightarrow{L(L^{2}(\mu))}$ defined by
\begin{equation}\label{19}\varphi(f,g)=M_{f\overline{g}}\end{equation}
where $M_{g}:L^{2}(\mu)\rightarrow{L^{2}(\mu)}$ is the
multiplication operator defined by
\begin{equation}\label{20}M_{g}(f)=g\cdot{f},\end{equation}
for $f\in{L^{2}}(\mu)$.

\end{cor}

\begin{prop}\label{18}
Let $X$ be a Hilbert space with respect to an $L(H)$-valued inner
product $\varphi:X\times{X}\rightarrow{L(H)}$ and
$F:X\rightarrow{L(H)}$ be an operator-valued norm defined by
equation (\ref{37}).

If $B(\subset{X})$ is a ($F$-norm) closed subspace of $X$, then
$B$ is also a Hilbert space with respect to
$\varphi|{B\times{B}}$.
\end{prop}
\begin{proof}
It is clear.
\end{proof}
\begin{thm}\label{60}
Let $(Y,\Omega,\mu)$ be a $\sigma$-finite measure space and
$L^{2}(Y,\Omega,\mu)=L^{2}(\mu)$.

Then, $C(Y)$ is a Hilbert space with respect to
$\varphi:Y\times{Y}\rightarrow{L(L^{2}(\mu))}$ defined by equation
(\ref{19}).

\end{thm}
\begin{proof}

Since the operator-valued function
$F_{\varphi}:L^{\infty}(\mu)\rightarrow{L(L^{2}(\mu))}$, defined
by
\[F_{\varphi}(f)=[\varphi(f,f)]^{1/2}=M_{|f|},\]
 is an operator-valued norm, by Corollary \ref{21}, and
 Proposition \ref{18}, it is enough to prove that
$C(X)$ is a $F_{\varphi}$-norm closed subspace of
$L^{\infty}(\mu)$.

Let $\{f_{n}\}_{n=1}^\infty$ be a sequence in $C(X)$ such that
$f_{n}\rightarrow{f}$ with respect to $F_{\varphi}$-norm. Thus,
for a given $\epsilon>0$, there is a positive integer
$N(\epsilon)>0$ such that for any $n\geq{N(\epsilon)}$,
\begin{equation}\label{23}
\norm{F_{\varphi}(f_{n}-f)}<\epsilon.
\end{equation}
Since
$\norm{F_{\varphi}(f_{n}-f)}=\norm{M_{|f_{n}-f|}}=\norm{f_{n}-f}_{\infty}$,
equation (\ref{23}) implies that $f_{n}\rightarrow{f}$ with
respect to the supremum norm. Since $(C(X),\norm{\texttt{
}}_{\infty})$ is a Banach space, $f\in{C(X)}$ which proves this
Theorem.

\end{proof}

------------------------------------------------------------------------

%\bibliographystyle{amsplain}
%\bibliography{xbib}
\end{document}